\documentclass[journal,twoside,web]{ieeecolor}

\usepackage{generic}
\usepackage{graphicx}
\def\BibTeX{{\rm B\kern-.05em{\sci\kern-.025em b}\kern-.08emT\kern-.1667em\lower.7ex\hbox{E}\kern-.125emX}}

\usepackage{cite}

\PassOptionsToPackage{hyphens}{url}\usepackage{hyperref}

\makeatletter
\newcommand{\removelatexerror}{\let\@latex@error\@gobble}
\makeatother

\usepackage{graphics} 
\usepackage{amsmath} 
\usepackage{amssymb}  
\usepackage{amsfonts}
\usepackage{tikz}
\usetikzlibrary{calc}

\usepackage{url}
\usepackage{physics}
\usepackage{pgfplots,tikz}
\pgfplotsset{compat=1.16}
\usepackage{textcomp}
\usepackage{mathtools}
\usepackage{textcomp}
\usepackage{algorithm}
\usepackage{algpseudocode} 
\usepackage{booktabs}

\algrenewcommand\algorithmicindent{1em}

\usetikzlibrary{plotmarks}
\usetikzlibrary{shapes,snakes}

\usepackage{comment}
\specialcomment{proofsketchcomment}{}{}
\specialcomment{proofcomment}{}{}

\newcommand{\mc}{\mathcal}
\newcommand{\ul}{\underline}
\newcommand{\ol}{\overline}

\newcommand{\defineas}{\coloneqq}

\newcommand{\D}{\mathcal{D}}
\renewcommand{\S}{\mathcal{S}}
\newcommand{\J}{\mathcal{J}}
\newcommand{\hJ}{\hat{\mathcal{J}}}
\newcommand{\T}{\mathcal{T}}

\newcommand{\plotjourney}[7]{
\draw[line width=1pt, fill=mycolor#2] (axis cs:#4,#3+0.1) rectangle (axis cs:#6,#3-0.1);

\draw[fill=mycolor#2!40] (axis cs:#6,#3+0.1) rectangle (axis cs:#5,#3-0.1);

\draw[line width=1pt] (axis cs:#6, #3+0.2) -- (axis cs:#6, #3-0.2);

\node[text width=0.1cm] at (axis cs:#4, #3+0.3){\footnotesize #1};

\node[diamond,draw,inner sep=1pt,fill=mycolor#2](n7) at (axis cs:#7, #3-0.5) {};

\draw (axis cs:#6, #3-0.2) -- (n7);
}

\DeclareMathOperator*{\argmax}{arg\,max}
\DeclareMathOperator*{\argmin}{arg\,min}

\newtheorem{definition}{Definition}
\newtheorem{lemma}{Lemma}
\newtheorem{thm}{Theorem}

\newtheorem{example}{Example}

\definecolor{mycolor1}{RGB}{230,97,1}
\definecolor{mycolor2}{RGB}{153,204,0}
\definecolor{mycolor3}{RGB}{253,184,99}
\definecolor{mycolor4}{RGB}{153,204,255}
\begin{document}

\title{\LARGE \bf Capacity-Constrained Urban Air Mobility Scheduling
}

\author{Qinshuang Wei, Gustav Nilsson, and Samuel Coogan
\thanks{%
The authors are with the School of Electrical and Computer Engineering, Georgia Institute of Technology, Atlanta, 30332, USA. {\tt\small \{qinshuang, gustav.nilsson, sam.coogan\}@gatech.edu}. S.~ Coogan is also with the School of Civil and Environmental Engineering, Georgia Institute of Technology. This work was supported in part by the National Science Foundation under grant \#1749357 and the Air Force Office of Scientific Research under award FA9550-19-1-0015.}}

\maketitle

\begin{abstract}

This paper studies the problem of scheduling urban air mobility trips when travel times are uncertain and capacity at destinations is limited. Urban air mobility, in which air transportation is used for relatively short trips within a city or region, is emerging as a possible component in future transportation networks. Destinations in urban air mobility networks, called vertiports or vertistops, typically have limited landing capacity, and, for safety, it must be guaranteed that an air vehicle will be able to land before it can be allowed to take off. We first present a tractable model of urban air mobility networks that accounts for limited landing capacity and uncertain travel times between destinations  with lower and upper travel time bounds. We then establish theoretical bounds on the achievable throughput of the network. Next, we present a tractable algorithm for scheduling trips to satisfy safety constraints and arrival deadlines. The algorithm allows for dynamically updating the schedule to accommodate, e.g., new demands over time.  The paper concludes with case studies that demonstrate the algorithm on two networks.

\end{abstract}

\section{Introduction}

There is growing interest in utilizing urban airspace for transportation of people and goods. Both commercial mobility-on-demand operators~\cite{2016uber} and government-sponsored research institutes such as NASA~\cite{thipphavong2018urban} are exploring such urban air mobility (UAM) solutions in cities and surrounding regions. Studies such as~\cite{balakrishnan2018blueprint, 2014NextGen,2018landscape, inrix, al2018identifying, ancel2017real} propose various approaches to allow urban air vehicles (UAVs) to travel safely and efficiently through cities. These proposed ideas cover a wide range of possibilities such as allowing UAVs to land at \emph{vertistops} or \emph{vertiports} installed on roofs of existing buildings or within cloverleaf exchanges on freeways. Several simulation tools~\cite{bosson2018simulation, xue2018fe3, 2019spark} have also been developed.

In this paper, we study a dynamic scheduling algorithm for UAM networks that accounts for uncertainty in travel time and limited landing capacity.
Trip demands, \emph{i.e.}, flights, must travel through designated routes and have arrival deadlines at their destinations. We then consider the problem of scheduling flight departures to ensure that all flights arrive no later than their deadlines and that there is always an available landing spot at the destination and intermediate nodes upon arrival. 
The main contributions are as follows: First, we present a model for UAM networks that allows for a time-varying set of trip demands, limited landing capacity at destinations, and uncertainty in travel times that is modeled nondeterministically with lower and upper travel time bounds. Because demands are time-varying, we refer to such networks as \emph{dynamic}; in the case when all demands are available at once, we call the network \emph{static} and refer to a \emph{static} scheduling problem.
Second, we present necessary conditions for the existence of a schedule for the dynamic UAM network. For the case of a static UAM network with a star-graph topology, we  show that this necessary condition is also sufficient. 
Third, we present a computationally efficient scheduling algorithm to compute a schedule satisfying safety constraints and arrival deadlines, and we demonstrate our approach on several case studies. 

This paper extends our earlier work \cite{ACC2021}, which considered the static scheduling problem for a limited class of star-like networks. In \cite{ACC2021}, we proposed evaluating the cost of a schedule as the sum of the difference between arrival and departure time for all trips, and we showed that an optimal schedule can be obtained from a mixed integer program. This approach extends to the more general scheduling problem in this paper, however, the mixed integer program quickly becomes intractable as the number of flights increases. Instead, the algorithm proposed in this paper uses a branch-and-bound heuristic that does not rely on a mixed integer formulation and therefore may only create a suboptimal schedule. However, we demonstrate through example that our algorithm is able to quickly obtain feasible schedules with reasonable cost, i.e., in under 1 second for 200 trips in two case study networks. Moreover, the algorithm computation can be continued in search of a lower cost schedule and interrupted at any point.

In the transportation scheduling literature, prior works have considered uncertain travel times or limitations on parking capacity separately. Particularly,~\cite{bulusu2018throughput} investigates how the flow of UAVs depends on the congestion level and finds through simulations similarities with ground highway traffic with high traffic density. 
The paper~\cite{sun2018dynamic} proposes an approach for traffic scheduling that can dynamically schedule flows on the link based on real-time link information.

Uncertainty in network routing problems has also been studied before for ground transportation. The paper~\cite{gendreau1996stochastic} provides a literature review of such problems, and examples of more recent work are presented in~\cite{dean2004algorithms}, which studies computation of minimum-cost paths through a time-varying network and considers several classes of waiting policies. In~\cite{samaranayake2011tractable}, a theoretical basis for optimal routing in transportation networks with highly varying traffic conditions is provided, where the goal is to maximize the probability of arriving on time at a destination given a departure time and a time budget. 

Limited parking availability has been considered in truck scheduling, where the drivers are usually required by law to park and rest after a specified amount of driving time. 
For example, in~\cite{ioannou2018optimizing}, it is assumed that 
truck drivers only have access to parking spots during specific time windows, but space limitations are not considered. The paper~\cite{kok2010dynamic} considers a similar problem with time-dependent travel times but does not take the availability of parking spots into account.

Scheduling problems have also been well-studied in the real-time systems community, e.g., in~\cite{albers2005efficient, fisher2007global}, where jobs often are assumed to arrive with a fixed periodicity and in some models have an uncertainty in their processing time. Our fundamental limits are similar in nature and compatible with those found in these works, but our results are tailored to applications in UAM networks. For example, for finite demands, we consider scheduling to achieve prescribed deadlines without excessively early departure times.

The remainder of this paper is organized as follows: In Section~\ref{sec:problem_formulation}, we present the dynamic UAM network model. In Section~\ref{sec:dynamic_sche_alg}, we 
establish necessary conditions for the existence of a feasible schedule. We also show that in certain star-like networks, the necessary condition becomes a sufficient condition. We  present an efficient algorithm for obtaining feasible schedules in  Section~\ref{sec:algorithm}, while some technical details about the algorithm are presented in the Appendix.
We then demonstrate in Section~\ref{sec:numerical} the proposed  algorithm on two case studies with up to 200 trips. In both networks, we are able to compute a reasonable schedule within 1 second.

\subsection{Notation}

We let $\mathbb{N}$ denote the natural numbers without zero while $\mathbb{N}_0$ the natural numbers with zero, and $\mathbb{R}$ the reals while $\mathbb{R}_+$ the positive reals. For a finite set $\mc A$, we let $\mathbb{R}^{\mc A}$, denote the set of vectors indexed by the elements in $\mc A$.

\section{Problem Formulation}
\label{sec:problem_formulation}
We model an urban air mobility (UAM) network with an acyclic\footnote{In practice, a directed graph can often be naturally decomposed into two or more acyclic graphs, e.g., flights inbound versus outbound of a city center. 
} directed graph $\mc G = (\mc V,\mc E)$, where $\mc V$ is the set of nodes and $\mc E$ is the set of links for the network. Nodes are physical landing sites for the UAVs, sometimes called \emph{vertistops} or \emph{vertiports}. Links are corridors of airspace connecting nodes. Each node $v \in \mc V$ has capacity $C_v \in \mathbb{N}_0$, that is, there are $C_v$ \emph{parking spots} at node $v$ where each parking spot allows at most one UAV to stay at any time. We denote the vector of capacities $C=\{C_v\}_{v\in\mathcal{V}}$.

We define $\tau: \mc E \rightarrow \mc V$ and $\sigma: \mc E \rightarrow \mc V$ so that for all $e=(v_1, v_2) \in \mc E$ where $v_1,v_2 \in \mc V$, $\tau(e) = v_1$ is the tail of edge $e$ and $\sigma(e) = v_2$ is the head of edge $e$. Let $S\subseteq \mathcal{V}$ (resp., $T\subseteq \mathcal{V}$) be the set of nodes that are not the head (resp., tail) of any edge, $S=\{v\in\mathcal{V}\mid \sigma(e)\neq v\ \forall e\in\mathcal{E}\}$ and $T=\{v\in\mathcal{V}\mid \tau(e)\neq v\ \forall e\in\mathcal{E}\}$. We assume $S\cap T=\emptyset$.

A \emph{route} $R$ is a sequence of connected links. Denote the number of links in route $R$ by $k_R$ and enumerate the links in the route  $1^R,2^R,\ldots,k_R^R$ and the nodes in the route $0^R,1^R,\ldots,k^R_R$. To avoid cumbersome notation, we use $\ell^R$ to denote both a link and its head node along a route, i.e., $\ell^R=\sigma(\ell^R)$ for all $\ell\in\{1,\dots,k_R\}$; the intended meaning will always be clear from context. Thus the route links and nodes are enumerated so that $0^R=\tau(1^R)$ is the origin node, $k_R^R$ is the destination node, and $\sigma(\ell^R) = \tau((\ell+1)^R)$ for all $\ell\in\{1,\dots,k_R\}$ ensures the sequence is connected. Further, when the route $R$ is clear from context, we drop the superscript-$R$ notation. We denote the set of nodes that $R$ travels through as $V(R)$.

We assume that, due to operational reasons, the UAVs are only allowed to travel along a set of routes $\mc R$.

Since, in reality, the travel time depends on external factors such as weather conditions, we assume that the travel time for each link is not exact, but rather bounded by a time interval. 
For each link $i \in \mc E$, let $\ol{x}_i$ and $\ul{x}_i$ with $\ol{x}_i \geq \ul{x}_i>0$ denote the maximum travel time and minimum travel time, respectively, for the link, and let $\ul{x}\in \mathbb{R}_+^\mathcal{E}$ and $\ol{x}\in \mathbb{R}_+^\mathcal{E}$ be the corresponding aggregated vectors. Once a UAV has landed at any node, it is assumed to block a landing spot for a fixed ground service time $w \in \mathbb{R}_+$. For ease of notation, we assume the ground service time is uniform at all nodes, but this assumption is straightforward to relax. 

\begin{definition}[UAM Network]
A UAM network $\mc N$ is a tuple $\mc N = (\mc G, C, \mc R, \ul{x}, \ol{x},w)$ where $\mc G, C,\mc R, \ul{x}, \ol{x}, w$ are the network graph, node capacities, routes, and minimum and maximum link travel times as defined above.
\end{definition}

To model the demand of UAV flights in a UAM network $\mc N = (\mc G, C, \mc R, \ul{x}, \ol{x},w)$, we assume that every flight is associated to a route $R\in \mc R$ and stops at intermediate nodes along the route. Therefore, a \emph{demand} is a pair $(R, f)$ where $R \in \mc R$ and $f \in \mathbb{R}_+$ is the latest time the UAV must arrive (i.e., land) at the destination $k^R_R\in\mathcal{V}$, \emph{i.e.}, its deadline.

A demand profile for a UAM network is a time-varying set $\mc D(t)=\{(R_j,f_j)\}_{j\in {\mc J(t)}}$ where ${\mc J(t)}$ is a time-varying index set. Then $\mc J(t)$, and, hence, $\mc D(t)$, is monotonically increasing and is assumed finite for all finite $t$, but may become infinite in the limit.  To coordinate the operation of the UAVs, a centralized scheduler aims to schedule all available demands $\mc D(t)$ at each time $t$. For example, new demands may become available in batches at certain times, requiring action from the scheduler, and other demands will reach their final destination, requiring no further consideration from the scheduler, although without loss of generality completed demands remain within the set $\mc D(t)$. Let $\mc D:=\cup _{t}\mc D(t)=\lim_{t\to \infty} \mc D(t)$ denote all demands that will ever be considered by the scheduler, and likewise let $\mc J:=\cup_{t}\mc J(t)=\lim_{t\to\infty} \mc J(t)$.

The scheduler associates to each demand a journey consisting of an assigned departure time and the set of realized arrival times along edges in the route. Therefore, a journey is updated over time as the UAV arrives at intermediate notes. Formally, for each $j\in {\mc J}(t)$, a \emph{journey} for demand $(R_j,f_j)\in {\mc D}(t)$  is a tuple $(A_j(t), \delta_{j})$, where $\delta_{j} \in \mathbb{R}$ is the departure time and is a decision variable of the scheduler, and $A_j(t)$ is the realized set of arrival times of the UAV at nodes along $R_j$ that have occurred by time $t$. Here, $\delta_j$ is a decision variable of the scheduler and $A_j$ is a result of realized travel times. In particular, when $R_j = \{\ell\}_{\ell=1}^{k_{R_j}}$,  $A_j(t) = \{A_{j,\ell}(t)\}_{\ell=1}^{k_{R_j}}$ where $A_{j,\ell}(t)$ is the arrival time at node ${\ell}$ if the arrival has occured by time $t$, and $A_{j,\ell}(t)=\infty$ if demand $j$ has not arrived at node $\ell$ by time $t$. 
Therefore, the departure time of the UAV from node $\ell$ is $\delta_{j,\ell} = A_{j,\ell} + w$ where we drop the explicit dependence on time since the equation is understood to hold only for times $t$ such that $A_{j,\ell}(t)<\infty$ and we let $\delta_{j,0} = \delta_j$.

For safety reasons, it is assumed that a UAV must be able to land immediately upon arrival at any node along its route. Whenever a UAV arrives at a node along the route of a journey, the journey will be updated accordingly, so that the uncertainty of the rest of the journey decreases. In particular, suppose the UAV departs some node $\ell_1-1$ along its route. The latest arrival time at some other node $\ell_2 \geq \ell_1$ along the route is denoted $a^{j}_{\ell_1,\ell_2}$ and given by
\begin{equation}
\label{eq:a_aggregate}
    a^{j}_{\ell_1,\ell_2} = \delta_{j,{\ell_1-1}} + \sum_{\ell=\ell_1}^{\ell_2} \ol{x}_{\ell}  + (\ell_2-\ell_1)w\,,
\end{equation}
i.e., $a^{j}_{\ell_1,\ell_2}$ is the departure time from node $\ell_1-1$ plus the upper bound of the time interval it takes to travel through the links $\{\ell\}_{\ell=\ell_1}^{\ell_2}$ with the time spent at each node. Note that it is time-dependent because $\delta_{j,{\ell_1-1}}$ is updated over time.
Further, the time interval that the UAV will potentially block a landing spot at node ${\ell_2}$ when departing from $\ell_1$ is given by
\begin{equation}
\label{eq:M_span}
\mc M_{\ell_1,\ell_2}^j = \left[ \delta_{j,{\ell_1-1}} + \sum_{\ell=\ell_1}^{\ell_2} \ul{x}_{\ell}+ (\ell_2-\ell_1)w, \, a^{j}_{\ell_1,\ell_2} + w \right] .
\end{equation}
We also define $m^{R_j}_{\ell_1,\ell_2}$ as the length of the time interval above, so that 
\begin{equation}
\label{eq:m_span}
    m^{R_j}_{\ell_1,\ell_2} = \sum_{\ell=\ell_1}^{\ell_2} \ol{x}_{\ell} -  \sum_{\ell=\ell_1}^{\ell_2} \ul{x}_{\ell} + w \,.
\end{equation}
Note that the superscript of $m$ is $R_j$ because it depends only on the route of demand $j$ and not its particular arrival and departure times.
We let $\mc M_{v_1,v_2}^j = \mc M_{\ell_1,\ell_2}^j$ if $v_1,v_2$ are two nodes along the route $R_j \in \mc R$, $v_1 = \ell_1^{R_j}$, $v_2 = \ell_2^{R_j}$ and $\ell_1 \leq \ell_2$. For all $\ell\in\{1,\ldots,k_R\}$, we let $a^{j}_{\ell} = a^{j}_{1,\ell}$, $\mc M_{\ell}^j = \mc M_{1,\ell}^j$ and  $m^{R_j}_{\ell} = m^{R_j}_{1,\ell}$. In the same manner, we let $a^j_v = a^{j}_{\ell}$, $\mc M_v^j = \mc M_\ell^j$ and $m^{R_j}_{v}=m^{R_j}_{\ell}$ if $v = \ell^{R_j}$.

The task, then, is to assign to each demand a journey such that capacity constraints and arrival times are satisfied. 

\begin{definition}[Schedules and Journeys]
\label{def:schedule}
Given a set of demands $\mc D (t_0)=\{(R_j,f_j)\}_{j\in {\mc J(t_0)}}$ at time $t_0$, a corresponding set of departure times $\mathcal{S}(t_0)=\{\delta_j\}_{j\in \hat{\mathcal{J}}(t_0)}$ with $\hat{\mathcal{J}}(t_0)\subseteq \mathcal{J}(t_0)$ and each $\delta_j\in\mathbb{R}$ is a \emph{schedule} for $\mc D(t_0)$ if:
\begin{enumerate}
\item The latest arrival time is not later than the deadline for all demands $j\in \hJ (t_0)$, that is, $a^j_{k_{R_j}} \leq f_j$ where we recall that $k_{R_j}$ is the final destination node for demand $j$ along route $R_j$ and $a^j_{k_{R_j}}$ is computed as in \eqref{eq:a_aggregate} which depends on any realized arrival/departure times at intermediate nodes;
\item the number of vehicles  at a node never exceeds capacity, \emph{i.e.}, for all $v\in{\mc V}$ and all $t \geq 0$,
\begin{equation*}
    \sum_{j:v \in V(R_j)}\mathbf{1}\left(t;\mc M_{v}^j(t_0)\right)\leq C_v
\end{equation*}
where the notation $\mathbf{1}(\cdot;\cdot)$ is an indicator such that $\mathbf{1}(t;[a,b])=1$ if $t\in [a,b]$ and $\mathbf{1}(t;[a,b])=0$ otherwise; and
\item In any finite time interval, only a finite number of UAVs depart.
\end{enumerate}
A schedule is a \emph{complete schedule} if $\hat{\mathcal{J}}(t_0)=\mathcal{J}(t_0)$; otherwise it is a \emph{partial schedule}. For any scheduled demand $j\in\hat{\mathcal{J}}(t_0)$, the pair $(A_j(t),\delta_j)$ is the \emph{journey} of demand $j$ where $A_j(t)$ is the set of realized arrival times as defined above.
\end{definition}
In the above definition, items 1 and 2 are the main features of a schedule, while item 3 is a mild technical requirement that is always satisfied in practice. In this paper, we consider the following problem.

\noindent \textbf{Scheduling Problem.} Given a set of demands $\D(t)$ and a sequence of  scheduling times $\T=\{t_0,t_1,t_2,\ldots\}$ that may be finite or infinite satisfying  $t_i<t_{i+1}$ for all $i$, at each scheduling time $t_i$, for the set of available demands $\D(t_i)$, determine $\hJ(t_i)\subseteq \J(t_i)$ the subset of flights to be scheduled and compute a schedule $\S(t_i)=\{\delta_j\}_{j\in\hJ(t_i)}$ satisfying the properties
\begin{enumerate}
\item $\hJ(t_i)\supseteq \hJ(t_{i-1})$ and $\S(t_i)=\S(t_{i-1})\cup \{\delta_j\}_{j\in \hJ(t_i)\backslash \hJ(t_{i-1})}$, and
\item For all $j\in \hJ(t_i)\backslash \hJ(t_{i-1})$, $\delta_j\geq t_i$
\end{enumerate}
 with $\J(t_{-1}) := \emptyset$ and $\S(t_{-1}):=\emptyset$.

\medskip
Property 1 of the above problem statement implies that once a demand is scheduled for departure, its departure time does not change at future scheduling times. Property 2 captures a causality requirement and implies that newly scheduled demands cannot have departure times in the past.

If $\D$ is time-invariant and only a single scheduling time $t_0=0$ is considered, the scheduling problem is called \emph{static}. This was the focus in our preliminary work \cite{ACC2021}. Otherwise, it is called a \emph{dynamic} scheduling problem.

Dynamic scheduling allows the schedule to be updated dynamically  at certain scheduling times. Schedule updates my be needed to accommodate new demands or because it may not be possible to schedule all known demands at some time, i.e., only a partial schedule can be obtained. Then, at future scheduling times after some flights arrive at intermediate nodes and reduce uncertainty in the remaining travel time, these previously unscheduled demands can be scheduled. Note that scheduling times need not be fixed in advance and can, for example, be triggered when a flight lands at an intermediate node.

We demonstrate how, even with a fixed and known set of demands, dynamic scheduling may be beneficial.

\begin{example}
\label{ex:2link}
\begin{figure}[h]
    \centering    
    \begin{tikzpicture}
    \node[draw, circle] (1) at (0, 1) {$v_1$};
    \node[draw, circle] (2) at (2, 1) {$v_2$};
    \node[draw, circle] (3) at (4, 1) {$v_3$};
    \draw[->] (1) -- node[above] {$[1,4]$} (2);
    \draw[->] (1) -- node[below] {$e_1$} (2);
\draw[->] (2) -- node[above] {$[2,3]$} (3);
    \draw[->] (2) -- node[below] {$e_2$} (3);
    \end{tikzpicture}
    
    \caption{A two-link network that is used to demonstrate the benefits of dynamic scheduling in Example~\ref{ex:2link}.}
    \label{fig:two_link}
\end{figure}
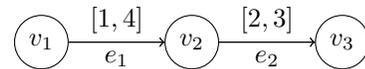

Consider a network with 3 nodes and 2 links as shown in Figure \ref{fig:two_link}. This network has a single route $R=\{e_1, e_2\}$. Suppose the travel time interval of link $e_1$ is $[1,4]$ and of link $e_2$ is $[2,3]$, the wait time at nodes $v_2$ and $v_3$ is $w=1$, and the capacity of nodes $v_2$ and $v_3$ are $C_2=C_3=1$. Consider the time-invariant set of demands $\D=\{(R,8),(R,11)\}$, \emph{i.e.}, there are two demands $\J=\{1,2\}$ with deadlines $f_1=8$ and $f_2=11$ that each take the only route $R$. A partial schedule at time $t=0$ is $\mc S(0)=\{\delta_1\}$ with $\delta_1=0$, the departure time for demand $j=1$. It is not possible to include a departure time for demand $j=2$ in schedule $\mc S(0)$ because of the uncertain travel time. Suppose, however, that demand $j=1$ arrives at node $v_2$ at time $t=2$. Then a new complete schedule at time $t=2$ is given by $\mc S(2)=\{\delta_1,\delta_2\}$ with $\delta_1=0$ and now $\delta_2=3$.
\end{example}

Even when the set of demands is time-invariant and finite, the static scheduling problem is computationally challenging because of the combinatorial decision in determining arrival order at nodes. In \cite{ACC2021}, we propose an integer program which exactly solves this static scheduling problem, but this approach is not suitable for large networks or when schedules are updated dynamically.

Instead, in this paper, we propose an algorithm based on branch-and-bound heuristics that considers at any time only the $K_0 \in \mathbb N$ most urgent demands that must be scheduled, where $K_0$ is a design parameter. Journeys are scheduled for these demands which determine the time intervals that the UAVs are potentially blocking any landing spot.
To avoid conflict, the blocking intervals cannot overlap, and hence these intervals are used for scheduling remaining demands. As the indeterminacy of traveling aggregates along the route and reduces whenever an intermediate node is reached, the schedule is dynamically updated over time to improve the efficiency of the UAM network. The detailed scheduling and updating method is explained in Section~\ref{sec:algorithm}. In the next section, we first derive fundamental theoretical limits for the scheduling problem.

\section{Theoretical Results for Scheduling}
\label{sec:dynamic_sche_alg}
In this section, we present  fundamental theoretical results on when complete schedules are available for a set of demands. We first consider the dynamic scheduling case and obtain necessary conditions for feasibility in the case of dynamic scheduling with finite number of demands at each scheduling time.
Then, we turn to the static scheduling problem and focus on the class of star networks representative of the case when several links lead to a central, main destination. In this case, we obtain sufficient and necessary conditions for scheduling an infinite set of demands representing, e.g., regular periodic flights. For more general network topologies, by considering a star sub-network consisting of a node and its neighbors, this result immediately leads to necessary conditions for feasibility.

\subsection{Necessary Condition for Dynamic Scheduling of Additional Demands}
\label{sec:dyn_thm}
Consider the general network setting defined in Section \ref{sec:problem_formulation}. Given a set of demands $\mc D(t)$ and a sequence of increasing scheduling times $t_0 < t_1 < t_2 < \dots$, we consider the set of available demands $\mc D(t_i) = \{(R_j,f_j)\}_{j\in\mc J(t_i)}$ and a schedule $\mc S(t_i) = \{\delta_j\}_{j\in \hJ(t_i)}$, for any $i = 0,1,\dots$, where $\hJ(t_i) \subseteq \J(t_i)$ and let $\mc J(t_{-1}) = \emptyset$ . 

We first introduce a cumulative departure function in time interval $[t_1, t_2]$ as $\Delta^R_v(t_1,t_2,\mc J) : \mathbb{R}^2 \rightarrow \mathbb{N}$ for all $v \in \mc V\backslash T$, where we recall the set of tail nodes $T$, so that $\Delta^R_v(t_1,t_2,\mc J)$ is the number of UAVs departing from node $v$ in the time interval $[t_1, t_2]$, while traveling through route $R\in \mc R$, i.e., 
\begin{equation}\Delta^R_v(t_1, t_2,\mc J)=|\{j\in {\mc J} \mid R_j=R \text{ and } \delta_{j,v} \in [t_1, t_2]\}|\,.\end{equation}

We then introduce a cumulative arrival function, $\Gamma^R_v(t_1,t_2,\mc J): \mathbb{R}^2 \rightarrow \mathbb{N}$ as the cumulative number of UAVs that travel through route $R\in \mc R$ and must arrive at the node $v \in V(R)\backslash\{0^R\}$ in the time interval $[t_1, t_2]$, i.e., 
\begin{equation}
\Gamma^R_v(t_1, t_2,\mc J)=|\{j\in {\mc J} \mid R_j = R\text{ and } \mc M_v^j \subseteq [t_1, t_2]  \}|. 
\end{equation}

We then define the flow rate at node $v$ through route $R$ in a finite interval $[t_1, t_2]$ as 
\begin{equation}
\label{eq:avg_flow_rate}
    r^R_v(t_1,t_2,\mc J) = \frac{\Delta^R_v(t_1,t_2,\mc J)}{t_2-t_1},
\end{equation}
where $v \in \mc V \backslash T$ and $R \in \mc R$.

Before exploring the necessary conditions of the schedules and demands, we need to define the bottleneck of a network, which is a basis of the necessary conditions and the algorithms.

\begin{definition}[s-t Cut]
An \emph{s-t cut} $\mc C = (S', T')$ is a partition of $\mc V$ such that $s \in S'$ and $t \in T'$, while $S \subseteq S'$ and $T \subseteq T'$.
\end{definition}

For an s-t cut $\mc C = (S', T')$ of the graph $\mc G$, we further define the set of edges with tails in $S'$ and heads in $T'$ as adjacent edges of the cut $\mc C$, $\mc E_{S'\to T'} = \{e\in \mc E: \tau(e) \in S' \text{ and } \sigma(e) \in T'\}$. We let $\tau(\mc E_{S'\to T'})$ (resp., $\sigma(\mc E_{S'\to T'})$) be the set of vertices that are tails (resp., heads) of all adjacency edges.

\begin{definition}[Flow and Maximizing Flow]
A set \{$\Tilde{r}^R_{v}\}_{v\in{\mc V},R\in \mc R}$ with each $\Tilde{r}^R_{v}\in\mathbb{R}_+$ is a \emph{flow} 
if the following 
are satisfied:
\begin{align}
\label{eq:f-1} \Tilde{r}^R_{v} &\geq 0 \quad & \forall& R \in\mc R, \, v \in \mc V\\ 
\label{eq:f-2} \Tilde{r}^R_{v} &= 0 \quad &\forall& R \in \mc R, \, v \in \mc V \setminus V(R)\\
\label{eq:f-3} \Tilde{r}^R_{v} &\leq {r}^R_{v} \quad &\forall& R \in \mc R, \, v \in \mc V \\
\label{eq:f-4}  \Tilde{r}^R_{v} &=  \Tilde{r}^R_{v+1} \quad  & \forall & R \in \mc R, \,  v \in \mc V\backslash T\\
\label{eq:f-5} \sum_{R \in \mc R} \Tilde{r}^R_{v-1} m^R_{v-1} &\leq C_v \quad & \forall & v \in \mc V\backslash S
\end{align}
where $r_{\ell_1,\ell_2}^{R} = C_{\ell_2}/m^{R}_{\ell_1,\ell_2}$ with $m^{R}_{\ell_1,\ell_2}$ as defined in~\eqref{eq:m_span} and we let $r_{\ell_2}^{R} = r_{1,\ell_2}^{R}$. Then $r_{\ell_1,\ell_2}^{R}$ is the maximum flow rate through the node $v$ along the route $R$.

A flow is a \emph{maximizing flow} if it maximizes $\sum_{R \in \mc R} \Tilde{r}^R_{k_R}$ over all flows.
\end{definition}

\begin{definition}[Bottleneck]
\label{def:BN_net}
The \emph{bottleneck} of a UAM network $\mc N = (\mc G, C,\mc R, \ul{x}, \ol{x},w)$ is a set of nodes $\bar{V}$ such that for all $R\in \mc R$, there exists a node  
$v\in \bar{V}$ and a maximizing flow \{$\Tilde{r}^R_{v}\}_{v\in{\mc V},R\in \mc R}$ such that $\Tilde{r}^R_{v} = {r}^R_{v}$ 
and that in the reduced graph $\mc G_{reduced} = (\mc V, \mc E \backslash (\mc E_1\cup \mc E_2) )$, for any $v_1 \in S$ and $v_2 \in T$, $v_1$ and $v_2$ are not connected, where $\mc E_1 = \{e\in \mc E \mid \tau(e) \in \bar{V}\}$ and $\mc E_2 = \{e\in \mc E \mid \sigma(e) \in \bar{V}\}$. 
\end{definition}

The lemma below provides an upper bound on the number of UAVs traveling through a node $v \in \mc V \backslash S$ that may be newly assigned in a schedule at any scheduling time.

\begin{lemma}
\label{lem:node_limit}
Consider a network $\mc N = (\mc G, C, \mc R, \ul x, \ol x,w)$ where $\mc G  = (\mc V, \mc E)$ and a sequence of increasing scheduling times $t_0 < t_1 <t_2 < \dots$. Given the set of demands $\D(t)$ and $R \in \mc R$, let $\mc S(t_i) = \{\delta_j\}_{j\in \mc \hJ(t_i)}$ be a schedule for the set of available demands $\mc D(t_i)$ at the scheduling time $t_i$ for all $i > 0$ where $\hJ(t_i) \subseteq \mc J(t_i)$ and let $\mc J(t_{-1}) = \emptyset$. Then $C_v(t_b-t_a) \geq \sum_{R \in \mc R_v} \Gamma^R_{v}(t_a,t_b,\hJ^\Delta_i)m^R_{v}$ for any time interval $[t_a, t_b]$ and for all $v \in \mc V\backslash S$, where $\mc R_v = \{R\in \mc R \mid v \in V(R) \}$ for all $v\in \mc V$, $\hJ^\Delta_i =\hJ(t_i) \backslash \hJ(t_{i-1})$, $t_0\leq t_a < t_b$ and $t_a, t_b \in \mathbb{R}$.
\end{lemma}
The proof of Lemma~\ref{lem:node_limit} follows from the Definition~\ref{def:schedule} and is omitted. We then deduce a necessary condition on nodes for given set of demands to have a complete schedule.
Lemma \ref{lem:node_limit} implies that, given a set of demands $\D(t_i)$ at the scheduling time $t_i$, $i>0$, if there exists a node $v \in \mc V \backslash S$ such that 
\begin{equation}
    C_v \cdot \left(\max_{j\in \mc J_v} f_{j} - t_i\right)< \sum_{j\in \mc J_v}m_v^{R_j}\, ,
\end{equation}
where $\mc J_v = \{j\in \J^\Delta_i \mid v\in V(R_j)\}$ and $\J^\Delta_i = \mc J(t_i)\backslash \hJ(t_{i-1})$,
then there does not exist a complete a schedule for $D(t_i)$. We can further narrow this necessary condition by denoting $\ul{\mc M}^{j}_{\ell_1,\ell_2}$ (resp., $\ol{\mc M}^{j}_{\ell_1,\ell_2}$) as the shortest time (resp., longest time) it takes to travel from node $\ell_1-1$ to $\ell_2$, so that $\ul{\mc M}^{j}_{\ell_1,\ell_2} = \sum_{\ell=\ell_1}^{\ell_2} \ul{x}_{\ell}+ (\ell_2-\ell_1)w $ and $\ol{\mc M}^{j}_{\ell_1,\ell_2} = \sum_{\ell=\ell_1}^{\ell_2} \ol{x}_{\ell}+ (\ell_2-\ell_1)w$. We denote $\ul{\mc M}^{j}_{\ell_2}  = \ul{\mc M}^{j}_{1,\ell_2}$ and $\ol{\mc M}^{j}_{\ell_2}  = \ol{\mc M}^{j}_{1,\ell_2}$.
We let 
\begin{equation}
\label{eq:f_j_v}
    f_{j,v} = f_j-\ol{\mc M}^{j}_{v}
\end{equation}
for all $j \in \mc J$ and $v\in V(R_j)$ so that $f_{j,v}$ is the latest arrival (resp., departure) time for demand $j$ at node $v$ when $v \neq 0$ (resp., $v=0$) such that the deadline at the destination $k_{R_j}$ is not violated. Therefore, for all $v \in \mc V \backslash S$, 
\begin{equation}
    C_v \cdot \left(\max_{j\in \mc J_v} f_{j,v} - \min_{j\in \mc J_v} \ul{\mc M}_v^j\right) \geq \sum_{j\in \mc J_v} m_v^{R_j}
\end{equation}
is a necessary condition for the set of demands $D(t_i)$ to have a complete schedule.

Theorem \ref{thm:network} below can help to come up with another necessary condition for the existence of a complete schedule for a set of demands at any scheduling time.  

\begin{thm}
\label{thm:network}
Consider a network $\mc N = (\mc G, C,\mc R, \ul x, \ol x,w)$ where $\mc G  = (\mc V, \mc E)$ and a sequence of increasing scheduling times $t_0 < t_1 <t_2 < \dots$. Given the set of demands $\mc D(t)$ and $R \in \mc R$, let $\mc S(t_i) = \{\delta_j\}_{j\in \mc \hJ(t_i)}$ be a schedule for the set of available demands $\mc D(t_i)$ at the scheduling time $t_i$, for any $i > 0$ and let $\J(t_{-1}) = \emptyset$.
If there exists a schedule $\mc S(t_i) = \{\delta_j\}_{j\in \hJ(t_i)}$, then it must satisfies:
\begin{equation}
    \label{eq:limit_net_BN}
     \sum_{R \in\mc R} \Delta^R_{0}(t_a,t_b,\hJ^\Delta_i)  \leq (t_b-t_a)\sum_{v\in \bar{V}}\sum_{R \in \mc R_v} {r}^R_v \,
\end{equation}
for any $t_a, t_b$ such that $t_0 < t_a \leq t_b$, where $\bar{V}$ is the bottleneck, $\hJ^\Delta_i =\hJ(t_i) \backslash \hJ(t_{i-1})$, and $\mc R_v = \{R\in \mc R \mid v \in V(R) \}$ for all $v\in \mc V$, i.e., the number of new departures assigned in the interval $[t_a, t_b]$ must be less than the product of the length of the time interval and the sum of maximum possible flow through the bottleneck of the network.
\end{thm}
\begin{proof}
We first rewrite \eqref{eq:limit_net_BN} as $\sum_{R \in \mc R} r^R_{0}(t_1,t_2,\hJ^\Delta_i)  \leq \sum_{v\in \bar{V}}\sum_{R \in \mc R_v} {r}^R_v$.
To complete the proof, we only need show that flow rate $r^R_v(t_1,t_2,\hJ^\Delta_i)$ at node $v$ through route $R$ is actually limited by the constraints \eqref{eq:f-1}--\eqref{eq:f-5}. Noticeably, we can consider it as a static scheduling while because the previous schedule will only reduce the number of UAVs to be scheduled depart in a time interval. Therefore, we assume there are no other assigned demands. 

First of all, consider $R\in \mc R$, then by definition it holds that $r^R_v(t_1,t_2,\hJ^\Delta_i) \geq 0$ for all $v\in \mc V$, and $r^R_v(t_1,t_2,\hJ^\Delta_i) = 0$ for all $v \notin V(R)$. By utilizing Lemma~\ref{lem:node_limit} it follows that $r^R_v(t_1,t_2,\hJ^\Delta_i) \leq r_v^R$ for all $R\in \mc R$ and $v\in \mc V$,  and $\sum_{R\in \mc R} r^R_{v-1}(t_1,t_2,\hJ^\Delta_i) m^R_{v-1} \leq C_v$ for all $v\in \mc V \backslash T$. 
In the static scheduling scenario we can consider $r_v^R(t_1,t_2,\hJ^\Delta_i)$ to be the static along the route, as a result, \eqref{eq:f-4} becomes trivial. Therefore, the flow rate $r^R_v(t_1,t_2)$ is actually limited by the constraints \eqref{eq:f-1}--\eqref{eq:f-5} and by Definition \ref{def:BN_net} of the bottleneck of a network, it follows that $\sum_{R \in \mc R} r^R_{0}(t_1,t_2,\hJ^\Delta_i)  \leq \sum_{v\in \bar{V}}\sum_{R \in \mc R_v} {r}^R_v$.
\end{proof}
\medskip
Similar to the necessary condition of each node, we can conclude from Theorem \ref{thm:network} that, for any scheduling time $t_i$, where $i>0$, there exists a complete schedule for the set of demands $\mc D(t_i)$ only if
\begin{equation}
    |\mc J| \leq  \left(\max_{j\in \J} (f_{j,0})-t_i\right) \cdot \sum_{v\in \bar{V}}\sum_{R \in \mc R_v} {r}^R_v \, .
\end{equation}

\subsection{Fundamental Limits for Static Scheduling}
\label{sec:acc_thm}
We investigate fundamental limits for static scheduling on the class of star graphs $\mc G_{*} = (\mc V_*,\mc E_*)$, where the set of nodes $\mc V_*$ consists of a central node $v_0$ and $L$ leaf nodes $v_{l}$ for $l=1,2,\hdots, L$. 
Then $\mc E_* = \{(v_{l}, v_{0})\mid  1\leq l\leq L\}$ is the set of links for the network. We label edge $(v_i,v_0)$ simply as edge $i$. We consider there to be $L$ routes in $\mc R$ and each route only consists of one edge.
We call such a network a \emph{local network} because it can be interpreted as a local portion of a larger network where we study only incoming flights to a particular node. 
Consider the static set of demands $\mc D = {(R_j,f_j)}_{j\in \mc J}$. We then define the collection of demands $\mc D$ as \emph{feasible} if there exists a schedule for $\mathcal{D}$.
If $\mathcal{D}$ is a finite collection of demands, then there always exists a schedule since departure times may be scheduled as early as needed to satisfy capacity constraints and desired arrival times. When $\mathcal{D}$ is an infinite collection of demands, then $\mathcal{D}$ may or may not be feasible. We assume that the number of deadlines within the time interval $[\tau, \tau+T]$ is finite for all $T < \infty$ and the limiting average number of deadlines in $[\tau, \tau+T]$ as $T\to \infty$ exists and is constant for varying $\tau\in \mathbb{R}$, \emph{e.g.}, $\mc D$ is a periodic set of demands.  
We next explore the fundamental limitations on the feasibility of the infinite set of demands $\mathcal{D}$ in the remainder of this section.

We introduce a cumulative departure function in time interval $[t_1, t_2]$ as $\Delta_v(t_1,t_2) : \mathbb{R}^2 \rightarrow \mathbb{N}$ for all $v \in \mc V\backslash\{v_0\}$ so that $\Delta_v(t_1,t_2)$ is the number of UAVs departing from node $v$ in the time interval $[t_1, t_2]$, i.e., 
$$\Delta_v(t_1, t_2)=|\{j\in {\mc J} \mid o_j=v\text{ and } \delta_j\in [t_1, t_2]\}|\,.$$

Given the cumulative departure function, we define the long-term average departure rate $r_v$ at node $v\in \mc V\backslash\{v_0\}$ as
\begin{equation}
\label{eq:r_v_local}
r_v \defineas \lim_{\substack{T\rightarrow +\infty}} \frac{1}{T}\, \Delta_v(\tau,\tau+T)  \,, \quad \forall \tau \in \mathbb{R}.
\end{equation}

In the same manner, we introduce a cumulative arrival function $\Gamma_v(t_1,t_2) : \mathbb{R}^2 \rightarrow \mathbb{N}$ as the cumulative number of UAVs that depart from origin $v$ and arrive at the destination in the time interval $[t_1, t_2]$, i.e., $\Gamma_v(t_1, t_2)=|\{j\in {\mc J} \mid o_j=v\text{ and vehicle arrives in }  [t_1, t_2]\}|$. In \cite[Lemma 1]{ACC2021}, it was established that for a local network, $\lim_{\substack{T\rightarrow +\infty}} \frac{1}{T}\, \Gamma_v(\tau,\tau+T) = r_v$ for all $ \tau \in \mathbb{R}$.

In the following theorem, we obtain a necessary and sufficient condition for the existence of a feasible schedule for a local network when the set of demands $\mc D$ is infinitely large, so that we can say immediately if there exists any feasible schedule. This will hence provide fundamental limits for how large demands a local network can handle.

\begin{thm}
\label{cor:multi_c}
Consider a local UAM network $\mc N$. 
An infinite set of demands $\mc D$ is feasible if and only if
\begin{equation}
\label{eq:nec_suf_cond}
 \sum_{1 \leq i\leq L} r_{v_i} \cdot(\ol{x}_i - \ul{x}_i+w) \leq C_{v_0} \,,
\end{equation}
where $r_{v_i}$ is as given in \eqref{eq:r_v_local} for all $v_i \in \mc V\backslash \{v_0\}$ and we recall $C_{v_0}$ is the capacity of the destination node $v_0$.
\end{thm}

The proof of Theorem \ref{cor:multi_c} is omitted but follows from the two lemmas below, which consider the special case when $C_{v_0}=1$. 
\begin{lemma}
\label{thm:neccessary}
Consider a local UAM network $\mc N$ with $C_{v_0} = 1$. If an infinite set of demands $\mc D$ is feasible, then
\begin{equation}
\label{eq:nec_cond}
 \sum_{1 \leq i\leq L} r_{v_i} \cdot(\ol{x}_i - \ul{x}_i+w) \leq 1 \,,
\end{equation}
where $r_{v_i}$ is as given in \eqref{eq:r_v_local} for all $v_i \in \mc V\backslash \{v_0\}$.
\end{lemma}

\begin{lemma}
\label{thm:sufficient}
Consider a local UAM network $\mc N$ with $C_{v_0} = 1$ and a countably infinite set of demands $\mc D=\{(o_j,f_j)\}_{j\in {\mc J}}$ with $f_j > t_0$ for all $j \in \mc J=\mathbb{N}$. 
If $\sum_{1 \leq i\leq L} r_{v_i} \cdot ( \ol{x}_i - \ul{x}_i+w) \leq 1$, then there always exists $\ol{M} \in \mathbb{R}$ such that
\begin{equation}
\label{eq:sufm_1}
    \sum_{1 \leq i\leq L}\Gamma_{v_i}(t_0,t_0+T)\cdot ( \ol{x}_i - \ul{x}_i+w) - f_T \leq \ol{M}
\end{equation}
for any $T > 0$, where $f_T = \max_{\{j\in \mc J\mid f_j\leq T\}}\{f_j\}$ is the last deadline that needs to be achieved before $T$.
In particular, this implies that the set of demands $\mc D$ is feasible.
\end{lemma}

The result \cite[Corollary 1]{ACC2021} in our preliminary work is a slight generalization of Theorem \ref{cor:multi_c} to the case with intermediate nodes along each branch of the star graph. Further, by considering a star sub-network consisting of a node and its neighbors, Theorem \ref{cor:multi_c} immediately leads to necessary conditions for feasibility of the infinity set of demands for networks with a general graph topology.

\begin{algorithm}[t]
\caption{Event-triggered Scheduling Algorithm}
\label{alg:dynamic}
\begin{algorithmic}[1]
\Require
UAM network $\mc N = (\mc G,C,\mc R,\ul{x},\ol{x}, w)$, 
maximum schedule size $K_0$
\State $\hJ(t_{-1}):=\emptyset$, $\S(t_{-1}):=\emptyset$, $i:=0$ , $t_0 = \Call{Now}{} \label{line:initialize} {}$ 
\State $\D(t_0):=\{(R_j,f_j)\}_{j\in {\mc J(t_0)}}$, $A(t_0) := \emptyset$, $\mc C(t_0) := \emptyset$ \Comment{Initialize demands, UAV arrival times and UAV parking spots}
\For{scheduling time $t_i$}
\State $\J_i^\Delta:=\J(t_i)\backslash \hJ(t_{i-1})$ \hfill \Comment{ indices of demands eligible to be scheduled}
\State $K:= \max\{K_0, \abs{\J_i^\Delta}\}$
\While{$K > 0$}
\State $(\D^\dagger,\J^\dagger):=$ subset of $K$ demands and indices with earliest deadlines from set $\{(R_j,f_j)\}_{j\in\J_i^\Delta}$
\State $A:= A(t_i)$, $\mc C:= \mc C(t_i)$
\State $OptSche := \Call{Insertion}{\D^\dagger,\S(t_{i-1}),t_i,A,\mc C,\mc N}$
\If{$OptSche \neq \S(t_{i-1})$}
\State $\S(t_i):=OptSche$ \hfill \Comment{feasible schedule found}
\State $\hJ(t_i):=\hJ(t_{i-1})\cup \J^\dagger$
\State \Call{Break}{}
\Else
\State $K := K-1$ 
\EndIf
\EndWhile
\State Update $\mc C(t_i)$
\State \Call{WaitForEvent}{} \Comment{wait for next event}
\State $i:=i+1$ 
\State $t_i = \Call{Now}{}$
\State Update $D(t_i)$ and $A(t_i)$
\EndFor
\end{algorithmic}
\end{algorithm}

\section{Efficient Algorithms For Dynamic Scheduling}
\label{sec:algorithm}

\begin{figure}[t!]
 \removelatexerror
\begin{algorithm}[H]
\caption{Schedule-Computing Algorithm}
\label{alg:schedule}
\begin{algorithmic}[1]
\Function{Insertion}{$\mc D, \S_{old}, t, A,\mc C,\mc N$}
\State \textbf{Inputs:} demands $\mc D=\{D_j\}_{j\in \mc J}$, existing schedule $\{\delta_j\}_{j\in \mc J_{old}} := \S_{old}$, time $t$, arrival times $A$, parking-spot assignments $\{\{c^j_v\}_{v\in V(R_j)}\}_{j\in \mc J(t)} = \mc C$, UAV network $((\mc V, \mc E), C,\mc R,\ul{x},\ol{x}, w): = \mc N$
\State $\ell^j := \argmin_{\ell \in V(R_j)}  a^j_{\ell,k_{R_j}}(t)$ $\quad \forall j\in \mc J_{old}$
\State $\Tilde{M}_{v,c} :=  \cup_{j\in \mc J_{old}|v = \ell^{R_j},\ell \geq \ell^j,c^j_{v}=c} \mc M^j_{\ell^j,v}$  $\forall v\in \mc V$, $c = 1,\dots,C_v$
\State $\Tilde{M} := \{\Tilde{M}_{v,c}\}_{v \in \mc V, c=1,\dots,C_v}$
\State  $\Tilde{M}_\text{max}: = \max_{v\in \mc V, c=1,\dots,C_v}(\sup \Tilde{M}_{v,c})$
\State $\mc D_1:=\{(R_j,f_j)\in\D\mid f_j-m^{R_j}_{k_{R_j}}+w\leq \Tilde{M}_\text{max}\}$, sort descending wrt deadline $f_j$
\State $\mc D_2:=\D\backslash {\mc D_1}$, sort ascending wrt $f_j$
\While{$\mc D_2 \neq \emptyset$}
\State $\S:=\S_{old}$ \Comment{initialization}
\For{$D_j \in \mc D_1$}
\State go through $\Tilde{M}$ and find the latest feasible departure time for $D_j$ which satisfies scheduling constraints at each node along $R_j$ and assign the available spot as $c^j_v$ for all $v\in V(R_j)$
\If{a feasible departure time $\delta_j$ is found}
\State $\S=\S_{old}\cup\{\delta_j\}$; update $\Tilde{M}$
\Else
\State \textbf{return}  $\S_{old}$ \hfill \Comment{No feasible schedule found}
\EndIf
\EndFor
\State  $\S_2:= \Call{prepare\_schedule}{\mc D_2}$
\If{$\S_2 \neq \emptyset$}
\State $\mc J_a  :=$ index set of $\mc D_2$ \label{alg_line:c_start}
\State $f^{node}_{v,c} := \max_{j\in \mc J} f_j+w, \forall v\in \mc V, 1\leq c \leq C_v$
\For{$\mc J_a \neq \emptyset$}
\State $j := \argmax_{j'\in \mc J_a} \sup \mc M^{j'}_{k_{R_{j'}}}$
\State $c^j_v := \argmax_{c = 1,\dots,C_v} f^{node}_{v,c} \quad \forall v\in V(R_j)$
\State $f^{node}_{v,c^j_v} := \inf{\mc M^j_v} \quad \forall v\in V(R_j)$
\State $\mc J_a  := \mc J_a \backslash \{j\}$
\EndFor \label{alg_line:c_end}
\State $\S = \S \cup \S_2$
\State \Return $\S$
\Else
\State remove the first journey $D_j \in \mc D_2$ and append to the head of $\mc D_1$
\EndIf
\EndWhile
\State \textbf{return} $\S_{old}$ \Comment{no feasible schedule found}
\EndFunction
\end{algorithmic}
\end{algorithm} 
\vspace{-1cm}
\end{figure}

In this section, we present an  algorithm for dynamic scheduling and, as a special case, static scheduling with finite demands. At its core, the algorithm creates a schedule at each scheduling time using branch-and-bound heuristics to efficiently determine candidate orderings for departure times and then a linear program to determine optimal departure times from the set of fixed orderings. We divide the algorithm into four algorithm blocks.  
The outermost block, Algorithm \ref{alg:dynamic}, creates schedules $\S(t_i)$ at scheduling times $\T=\{t_0,t_1,t_2,\ldots\}$ given a UAM network $\mathcal{N}$, time-varying demands $\D(t)$, and UAV arrival times $A(t)$ that are updated according to the description in Section \ref{sec:problem_formulation}. We introduce $\mc C_j = \{c^j_\ell\}_{\ell=1}^{k_{R_j}}$ as the set of parking spots occupied by the UAV of the $j$'th journey along its route, where $c^j_\ell \in \{1,\dots,C_{\ell^{R_j}}\}$ is the parking spot at node ${\ell}$ that the UAV occupies upon its arrival. Therefore $\mc C(t) = \{\mc C_j\}_{j\in \mc J(t)}$ is updated while scheduling. 
The scheduling times $\T$ correspond to when new demands become available to the scheduler or when a UAV lands at a node along its journey, and these times are generally not known in advance. In this way, Algorithm \ref{alg:dynamic} acts as an event-triggered algorithm that creates a new schedule each time a vehicle lands or a new demand is introduced. For computational considerations, Algorithm 1 also allows for considering only the first $K_0$ UAV flights with the earliest deadline, and other available demands are scheduled at later scheduling times. In practice, we found that including such a ceiling significantly increases computational speed with negligible effect on the final schedules for all demands.

It is possible that a complete schedule accommodating worst case travel times may not be possible at each scheduling event. In this case, Algorithm \ref{alg:dynamic} creates 
a partial schedule. 
The unassigned demands are then considered at the next scheduling time, and any demand that cannot be fulfilled anymore is ultimately dropped. 

Algorithm \ref{alg:dynamic} seeks a feasible partial or complete schedule minimizing the \emph{Sum of Difference (SoD)} cost
\begin{equation}
\label{eq:sod}
    \text{SoD}=\sum_{j\in \hJ} (f_j-\delta_j)
\end{equation}
where $\hJ$ is the index set for the schedule. We refer to this cost as the SoD cost since it is the sum of the difference between arrival time and departure time for the flights. This cost corresponds to the typical preferences of customers to leave no earlier than necessary while still guaranteeing arrival by a desired deadline. An immediate lower bound for the SoD cost is obtained by considering worst case travel times along the routes for all demands $\mc{\hat{J}}$, ignoring capacity constraints at nodes.

With relatively few demands and a simple network topology, e.g., a star network as considered in \cite{ACC2021}, the exact optimal schedule minimizing \eqref{eq:sod} can be obtained from a  mixed-integer program. However, for general acyclic directed graphs and/or a large set of demands, obtaining a schedule from a mixed-integer program is not computationally tractable. We next propose an efficient but possibly suboptimal approach for these cases 
in Algorithm~\ref{alg:schedule}--\ref{alg:order_find}.

Algorithm \ref{alg:schedule} considers a set of demands that need to be scheduled and partitions them into two sets according to their deadlines and existing schedule. 
To achieve this, the algorithm first computes the time intervals that UAVs will potentially occupy each parking spot of each node according to \eqref{eq:M_span} and the parking spot assignment $\mc C(t)$. This is time-varying since it depends on time-varying arrivals $A_j(t)$, and we let $\Tilde{M}(t)$ collect all of the blocked time intervals computed by~\eqref{eq:M_span} for all journeys at all parking spots of all nodes at time~$t$. When a UAV lands at a node along its route, the time interval that the UAV will potentially occupy any node along the rest of its route reduces since the uncertainty of traveling is reduced, so that we can update $A(t)$ and $\Tilde{M}(t)$ accordingly. When journeys are assigned to new demands at time $t$, $\Tilde{M}(t)$ is used to avoid parking spot conflicts. 
Using $\Tilde{M}(t)$, Algorithm~\ref{alg:schedule} divides demands into those with earlier deadlines, $\D_1$, which are inserted into the current schedule if possible while the demands with later deadlines, $\D_2$, are scheduled using  Algorithm \ref{alg:schedule_1} beyond the latest time of the existing schedule. If a feasible complete schedule for $\D_2$ cannot be found, the demand with the earliest deadline is moved from $\D_2$ to $\D_1$ and the process repeats.

Algorithms \ref{alg:schedule_1} and \ref{alg:order_find} are presented in the Appendix with MATLAB-like syntax. Algorithm \ref{alg:schedule_1} prepares the variables for Algorithm \ref{alg:order_find}, picks the best solution returned from Algorithm \ref{alg:order_find}, and transforms that into a schedule. Algorithm \ref{alg:order_find} is a branch-and-bound algorithm 
that finds a list of possible schedules for the given demands. While the algorithm is inspired by the classical Bratley's algorithm for task scheduling~\cite{buttazzo2011hard}, several adoptions have been made for the problem we are addressing. For instance,  
Bratley's algorithm generally seeks a single schedule with the earliest possible departure times, while  Algorithm \ref{alg:order_find} returns a set of possible departure times aiming to minimize the SoD cost \eqref{eq:sod}.

Algorithm~\ref{alg:order_find} also employs a pruning technique different than Bratley's algorithm. Consider the set of candidate schedules as a rooted tree. The algorithm starts from scheduling the last journey for the unassigned demands by visiting the demands in descending order according to their deadlines, picking a demand, recording the latest possible departure time and removing the demand from the unassigned list. The process is repeated until all demands are assigned or the branch is stopped by the pruning mechanism. If all demands are assigned, the current branch will be stored as a possible schedule and the algorithm will continue with another branch until all branches are considered (visited or pruned). The algorithm then returns all stored schedules.

We suggest several pruning mechanisms to avoid an exhaustive search, with the first two inspired by Bratley's Algorithm:
\begin{enumerate}
    \item If an unassigned demand cannot be scheduled according to the assigned schedules without considering other unassigned demands, the current branch will be discarded.
    \item If any node cannot afford the unassigned demands according to Lemma \ref{lem:node_limit} with the assigned schedules, the current branch will be pruned.
    \item If the demand with latest deadline can be scheduled without interfering with any other demand, then it is scheduled as the last journey to arrive at the destination. 
    \item If two demands choose the same route, their order follows the order of their deadline, i.e., if $D_{j_1},D_{j_2} \in \mc D_{in}$ and $R_{j_1} = R_{j_2}$, while $f_{j_1}<f_{j_2}$, then only the branches with $D_{j_1}$ arriving earlier then $D_{j_2}$ will be preserved.
    \item When searching along a branch, the current scheduled demands are compared with stored schedules. If the SoD cost along the current branch must exceed that of any of the stored schedules, the current branch will be abandoned. 
\end{enumerate}
Even with the above pruning techniques, Algorithm \ref{alg:order_find} may still produce a large set of feasible schedules. Although there can be several branches to explore, it is possible to stop the search at any point, given that at least one branch has been found, to get a timely but sub-optimal solution.

\section{Numerical Study}
\label{sec:numerical}
In this section, we demonstrate our algorithm on two case studies with up to 200 UAVs each. We first demonstrate the  algorithm  on an eight-node network with dynamic scheduling updates. In the second case study, we consider a static scheduling example for a network of the Atlanta region that appeared in our prior work~\cite{ACC2021}. For this case, we show that the proposed algorithm generates a schedule with a nearly optimal cost in considerably less computational time than the exact integer program proposed in \cite{ACC2021}. In both case studies, the proposed algorithm obtained reasonable schedules within 1 second.

\subsection{Case Study 1: Dynamic Scheduling}
\label{sec:case_study}
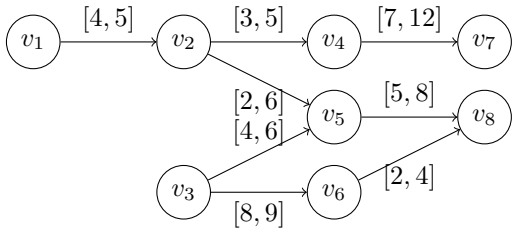
\begin{figure}[h]
    \centering    
    \begin{tikzpicture}
    
    \node[draw, circle] (1) at (0, 1) {$v_1$};
    \node[draw, circle] (2) at (2, 1) {$v_2$};
    \node[draw, circle] (4) at (4, 1) {$v_4$};
    \node[draw, circle] (7) at (6, 1) {$v_7$};

    \node[draw, circle] (5) at (4, 0) {$v_5$};
    \node[draw, circle] (8) at (6, 0) {$v_8$};
    
    \node[draw, circle] (3) at (2, -1) {$v_3$};
    \node[draw, circle] (6) at (4, -1) {$v_6$};

    \draw[->] (1) -- node[above] {$[4,5]$} (2);
    \draw[->] (2) -- node[above] {$[3,5]$} (4);
    \draw[->] (4) -- node[above] {$[7,12]$} (7);
    \draw[->] (2) -- node[below] {$[2,6]$} (5);
    \draw[->] (5) -- node[above] {$[5,8]$} (8);
    \draw[->] (3) -- node[above] {$[4,6]$} (5);
    \draw[->] (3) -- node[below] {$[8,9]$} (6);
    \draw[->] (6) -- node[below] {$[2,4]$} (8);

    \end{tikzpicture}
    
    \caption{A network with $8$ nodes and $8$ links that is used to illustrate the case study in \ref{sec:case_study}.}
    \label{fig:case_study_net}
\end{figure}

The first case study considers the network in Fig. \ref{fig:case_study_net}.   
The set of all possible origins (resp., destinations) is $S = \{v_1, v_3 \}$ (resp., $T = \{v_7, v_8 \}$). We assume the  origins $v_1, v_3$ do not have capacity constraints, while $C_{v_2} = 1$, $C_{v_4} = 2$, $C_{v_5} = 3$, $C_{v_6} = 1$, $C_{v_7} = 3$ and $C_{v_8} = 5$. The links are indicated in the figure and the corresponding travel time intervals are labeled beside the links. We consider four routes $\mc R = \{R^1,R^2,R^3,R^4\}$ with $ R^1 = \{(v_1,v_2),(v_2,v_4),(v_4,v_7)\}$, $R^2 = \{(v_1,v_2),(v_2,v_5),(v_5,v_8)\}$, $R^3 = \{(v_3,v_5),(v_5,v_8)\}$ and $R^4 = \{(v_3,v_6),(v_6,v_8)\}$.
Each UAV remains at the vertistops along its path for $w= 1$ minute after landing (all times are given in minutes). 
Algorithms \ref{alg:dynamic}, \ref{alg:schedule}, \ref{alg:schedule_1}, and~\ref{alg:order_find} are implemented in MATLAB to obtain a schedule $\mathcal{S}=\{(\delta_j)\}_{j\in{\mc J}}$. 

In simulation, realized travel times are uniformly randomly drawn from the travel time intervals. We first consider 43 randomly generated demands. Subsets of demands become available for scheduling across ten scheduling times $\T=\{0,15,30,45,50,60,80,85,100,120\}$.

Fig. \ref{fig:ex_node_8} demonstrates the resulting schedule. The figure compares the scheduled and actual arrivals at node $8$ in the time interval $[100, 160]$ minutes as observed at time $t=100$ minutes (top) and $t=120$ minutes (bottom). The green, orange and blue bars represent the reserved time slots for UAVs traveling through $R^2$, $R^3$ and $R^4$, respectively. UAVs on route $R^1$ are not shown because $R^1$ does not go through node $8$. We label an ID number above each bar to identify the UAVs. The diamond represents the arrival deadline of a demand; the solid-color bar represents the time interval that the UAV could possibly arrive given the schedule and uncertain travel times. The UAV will then stay at the node for time $w$ after arrival, which is represented by the lightly shaded bars in the figure. Therefore, the entire bar represents the time interval the UAV may appear at node $8$ under best and worst case travel times. Comparing the top and the bottom figures, we observe several representative changes: the journey $30$ has not yet arrived at destination by time $100$ in the top plot but is completed by time $120$ in the bottom plot; the time slot reserved for journey $25$ has shrunk because the UAV arrived at an intermediate node along its route in the time interval $[100,120]$ minute so that the uncertainty of its arrival at node $8$ reduces; and journeys $37$--$43$ are newly scheduled in the bottom plot.
\begin{figure}[!t]
\centering
\begin{tikzpicture}
\begin{axis}[%
width=7cm,
height=3.5cm,
scale only axis,
xmin=100,
xmax=160,
xlabel={Time (min)},
ymin=0,
ymax=5.5,
ytick={1,2,3,4,5},
yticklabels = {1,2,3,4,5},
ylabel={Slot},
ylabel near ticks,
title style={align=center,  yshift=-8pt},
title={Schedule for node 8 at time $t=100$}]

\plotjourney{1}{3}{1}{19.123}{20.123}{19.123}{19.776}
\plotjourney{2}{2}{1}{47.077}{48.077}{47.077}{49.399}
\plotjourney{4}{4}{3}{39.696}{40.696}{39.696}{39.809}
\plotjourney{5}{2}{5}{29.973}{30.973}{29.973}{33.986}
\plotjourney{8}{2}{4}{34.642}{35.642}{34.642}{38.953}
\plotjourney{9}{3}{2}{44.686}{45.686}{44.686}{45.186}
\plotjourney{10}{4}{2}{15.902}{16.902}{15.902}{16.708}
\plotjourney{11}{4}{1}{36.337}{37.337}{36.337}{40.187}
\plotjourney{12}{2}{3}{47.082}{48.082}{47.082}{51.624}
\plotjourney{15}{3}{2}{69.875}{70.875}{69.875}{71.390}
\plotjourney{16}{4}{1}{69.579}{70.579}{69.579}{71.667}
\plotjourney{17}{3}{3}{69.727}{70.727}{69.727}{71.966}
\plotjourney{18}{3}{4}{65.100}{66.100}{65.100}{67.949}
\plotjourney{19}{4}{2}{76.423}{77.423}{76.423}{78.042}
\plotjourney{20}{2}{1}{90.374}{91.374}{90.374}{92.626}
\plotjourney{21}{2}{2}{84.042}{85.042}{84.042}{87.081}
\plotjourney{22}{3}{3}{88.309}{89.309}{88.309}{95.045}
\plotjourney{23}{4}{1}{75.335}{76.335}{75.335}{75.917}
\plotjourney{24}{4}{1}{106.059}{110.059}{109.059}{109.059}
\plotjourney{25}{2}{1}{121.222}{130.222}{129.222}{129.222}
\plotjourney{26}{3}{5}{112.463}{118.463}{117.463}{117.463}
\plotjourney{27}{3}{2}{111.469}{117.469}{116.469}{116.469}
\plotjourney{29}{3}{3}{117.986}{123.986}{122.986}{122.986}
\plotjourney{30}{3}{4}{115.292}{121.292}{120.292}{120.292}
\plotjourney{31}{3}{2}{122.795}{128.795}{127.795}{127.795}
\plotjourney{32}{3}{4}{122.916}{128.916}{127.916}{127.916}
\plotjourney{33}{2}{1}{140.087}{149.087}{148.087}{148.087}
\plotjourney{34}{4}{3}{125.222}{129.222}{128.222}{130.777}
\plotjourney{35}{4}{5}{123.222}{127.222}{126.222}{130.127}
\plotjourney{36}{3}{2}{132.184}{138.184}{137.184}{137.184}

\end{axis}
\end{tikzpicture}%
\begin{tikzpicture}
\begin{axis}[%
width=7cm,
height=3.5cm,
scale only axis,
xmin=100,
xmax=160,
xlabel={Time (min)},
ymin=0,
ymax=5.5,
ytick={1,2,3,4,5},
yticklabels = {1,2,3,4,5},
ylabel={Slot},
ylabel near ticks,
title style={align=center,  yshift=-8pt},
title={Schedule for node 8 at time $t=120$}]

\plotjourney{1}{3}{1}{19.123}{20.123}{19.123}{19.776}
\plotjourney{2}{2}{1}{47.077}{48.077}{47.077}{49.399}
\plotjourney{4}{4}{3}{39.696}{40.696}{39.696}{39.809}
\plotjourney{5}{2}{5}{29.973}{30.973}{29.973}{33.986}
\plotjourney{8}{2}{4}{34.642}{35.642}{34.642}{38.953}
\plotjourney{9}{3}{2}{44.686}{45.686}{44.686}{45.186}
\plotjourney{10}{4}{2}{15.902}{16.902}{15.902}{16.708}
\plotjourney{11}{4}{1}{36.337}{37.337}{36.337}{40.187}
\plotjourney{12}{2}{3}{47.082}{48.082}{47.082}{51.624}
\plotjourney{15}{3}{2}{69.875}{70.875}{69.875}{71.390}
\plotjourney{16}{4}{1}{69.579}{70.579}{69.579}{71.667}
\plotjourney{17}{3}{3}{69.727}{70.727}{69.727}{71.966}
\plotjourney{18}{3}{4}{65.100}{66.100}{65.100}{67.949}
\plotjourney{19}{4}{2}{76.423}{77.423}{76.423}{78.042}
\plotjourney{20}{2}{1}{90.374}{91.374}{90.374}{92.626}
\plotjourney{21}{2}{2}{84.042}{85.042}{84.042}{87.081}
\plotjourney{22}{3}{3}{88.309}{89.309}{88.309}{95.045}
\plotjourney{23}{4}{1}{75.335}{76.335}{75.335}{75.917}
\plotjourney{24}{4}{1}{107.570}{108.570}{107.570}{109.059}
\plotjourney{25}{2}{1}{124.451}{128.451}{127.451}{129.222}
\plotjourney{26}{3}{5}{115.572}{116.572}{115.572}{117.463}
\plotjourney{27}{3}{2}{116.128}{117.128}{116.128}{116.469}
\plotjourney{29}{3}{3}{118.501}{122.501}{121.501}{122.986}
\plotjourney{30}{3}{4}{118.243}{119.243}{118.243}{120.292}
\plotjourney{31}{3}{2}{123.282}{127.282}{126.282}{127.795}
\plotjourney{32}{3}{4}{123.616}{127.616}{126.616}{127.916}
\plotjourney{33}{2}{1}{140.087}{149.087}{148.087}{148.087}
\plotjourney{34}{4}{3}{125.222}{129.222}{128.222}{130.777}
\plotjourney{35}{4}{5}{123.222}{127.222}{126.222}{130.127}
\plotjourney{36}{3}{2}{132.184}{138.184}{137.184}{137.184}
\plotjourney{37}{3}{3}{135.087}{141.087}{140.087}{147.674}
\plotjourney{38}{3}{1}{164.273}{170.273}{169.273}{169.273}
\plotjourney{39}{3}{2}{143.087}{149.087}{148.087}{150.023}
\plotjourney{40}{4}{3}{144.087}{148.087}{147.087}{148.542}
\plotjourney{42}{2}{4}{138.087}{147.087}{146.087}{150.077}
\plotjourney{43}{3}{3}{148.939}{154.939}{153.939}{153.939}

\end{axis}
\end{tikzpicture}%
\caption{Scheduled/realized occupation of the five parking spots at Node 8 in the time interval $[100,160]$ determined at time $100$ (top) and time $120$ (bottom). The green, orange and blue bars represent the reserved time slots for UAVs traveling through $R^2$, $R^3$ and $R^4$, respectively. An ID number above each bar identifies the UAVs. The diamond represents the arrival deadline of a journey; the solid-color bar represents the time interval that the UAV is scheduled to arrive. The UAV stays at the node for time $w=1$ after arrival (represented by the lightly shaded bar). Therefore, the entire bar represents the time interval the UAV may or did appear at node $8$. Comparing the top and the bottom figures, we can observe several representative changes: the journey $30$ has not yet arrived at the destination in the top plot computed at time 100 but is completed in the bottom plot computed at time 120; the time slot reserved for journey $25$ has shrunk because the UAV arrived at an intermediate node in the time interval $[100,120]$ so that the uncertainty of its arrival at node $8$ reduced; journeys $37$-$43$ are newly scheduled.}
\label{fig:ex_node_8}
\end{figure}
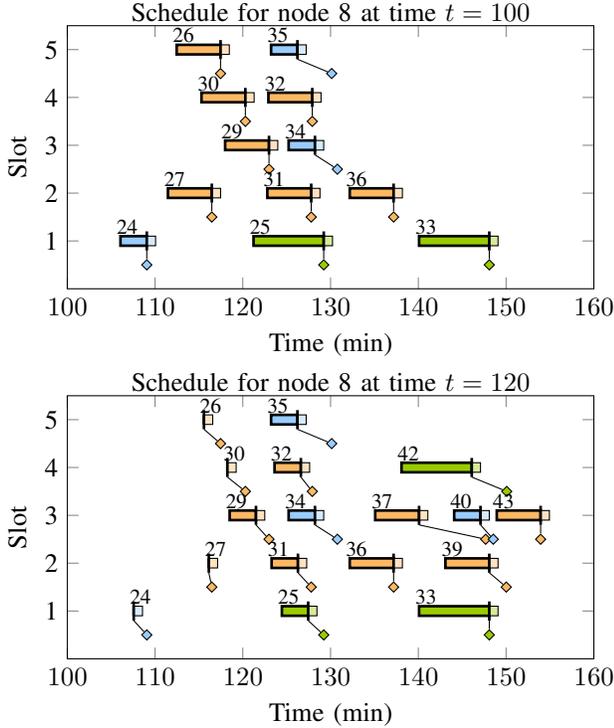

A schedule for  the 43 UAVs described above is obtained in 0.8 seconds with a SoD cost of 759.3 minutes.  By considering worst case travel times for all demands but not capacity constraints, the lower bound for the SoD cost is 746 minutes. Thus, the algorithm produces a schedule that is at least within 2\% of optimal, and since an exact optimal schedule is not computationally tractable, it is unknown exactly how far from optimal the obtained scheduled is, or even if it is, in fact, optimal itself.

Next, we consider a larger set of demands with size 200. The routes are chosen randomly among the four routes and the deadlines are picked randomly in the interval $[40, 1540]$ minutes. The algorithm finds its first feasible schedule at 0.2801 seconds with SoD cost equal to 3759.9 minutes, which is again within 2\% of the lower bound of 3693 minutes. 
\subsection{Case Study 2: Atlanta Network}
\color{black}

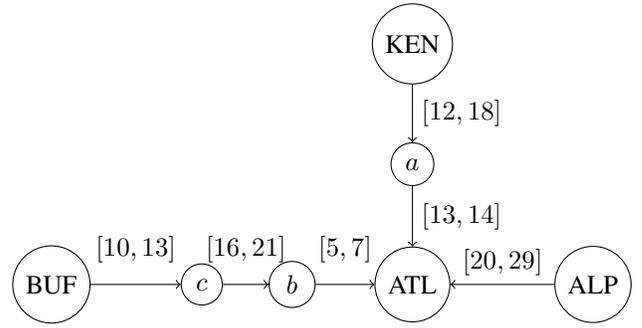
\begin{figure}
    \centering
    \begin{tikzpicture}[scale=0.8]
    \node[draw, circle] (0) at (0,0) {ATL};
    
    \node[draw,circle] (1) at (3,0) {ALP};
    \draw[->] (1) --node[above] {$[20,29]$} (0);
    
    \node[draw,circle] (2a) at (0,2) {$a$};
    \draw[->] (2a) --node[right] {$[13,14]$} (0);
    \node[draw,circle] (2) at (0,4) {KEN};
    \draw[->] (2) --node[right] {$[12,18]$} (2a);
    
    \node[draw,circle] (3a) at (-2,0) {$b$};
    \draw[->] (3a) --node[above=5pt] {$[5,7]$} (0);
    \node[draw,circle] (3b) at (-3.5,0) {$c$};
    \draw[->] (3b) --node[above=5pt] {$[16,21]$} (3a);
    \node[draw,circle] (3) at (-6,0) {BUF};
    \draw[->] (3) --node[above=5pt] {$[10,13]$} (3b);
    \end{tikzpicture}
    \caption{A graph for a UAM network of Atlanta (ATL) with three exurbs: Alpharetta (ALP), Kennesaw (KEN) and Buford (BUF). We consider Atlanta as the central node, and there are $0,1,2$ intermediate nodes between Atlanta node and the three leaf nodes ``ALP", ``KEN" and ``BUF", respectively. The time interval need for traveling through each link is labeled beside the corresponding link.}
    \label{fig:atl_example}
\end{figure}

We next consider the network shown in Fig. \ref{fig:atl_example} that appears also in \cite{ACC2021}.  
This network is inspried by a recent report by INRIX~\cite{inrix} which suggests that a UAM local network traveling to the city of Atlanta from three exurbs, Alpharetta, Kennesaw and Buford, has the potential to offer significant time savings compared to ground transportation during peak travel times. We further assume there is an intermediate vertistop $a$ at a suburb between Kennesaw and Atlanta, and intermediate vertistops $b$ and $c$ at suburbs between Buford and Atlanta, as shown in the figure. The corresponding travel time intervals are labeled beside the links.
We take Atlanta as node number $0$ and the exurbs Alpharetta, Kennesaw and Buford $1$, $2$ and $3$, respectively. We let ${\mc R} = \{R^1, R^2, R^3\}$, where $ R^1 =\{(v_1,v_0)\}$, $ R^2 =\{(v_2,a), (a,v_0)\}$, and $\mc R^3 =\{(v_3,c),(c,b),(b,v_0)\}$.

We now consider demands $\mathcal{D}=\{(R_j,f_j)\}_{j\in {\mc J}}$ defined subsequently, where $R_j \in \{R^1,R^2,R^3\}$. Note that the destination for all demands is Atlanta, node $v_0$. We consider a time horizon of three hours ($T = 180$ minutes) and assume there are two landing spots in Atlanta, but only one landing spot at vertistops $a,b$ and $c$, i.e., $C_{v_0} = 2$, $C_{a}=C_{b}=C_{c}=1$. Each UAV stays at the intermediate vertistops along its path for $w_I= 1$ minute and at the Atlanta vertiport for $w=5$ minutes after landing.

We assume the number of UAV demands across the three origins is $[h_1,h_2,h_3] = [4,4,19]$ and that 
arrival deadlines are set at regular intervals, i.e., if origin $v_i$ is tasked with sending $h_i$ UAVs to Atlanta, the deadlines are ${\left\lfloor \frac{180}{1+h_i}\cdot k + 0.5 \right\rfloor}$, $k = 1,2,\ldots, h_i$. 

For comparison, the mixed-integer optimization problem in \cite{ACC2021} is solved in MATLAB using Gurobi with YALMIP toolbox to obtain a schedule that exactly minimizes the SoD cost \eqref{eq:sod}. Using the exact mixed-integer program from \cite{ACC2021} takes around $150$ minutes of computation time on a standard laptop, and the optimal SoD cost as in \eqref{eq:sod} is 1587 minutes. Note that the lower bound computed using worst cast travel time but ignoring capacity constraints is 1173 minutes. Thus, the capacity constraints at the nodes play a significant role in overall network SoD cost.

In contrast, the algorithm proposed in Section \ref{sec:algorithm} obtains a feasible schedule in 0.05 seconds. The SoD cost of this sub-optimal schedule is 1669 minutes. The SoD cost reduces to 1605 minutes when the algorithm finds an alternative schedule after 5.8 seconds of computation time. 
Fig. \ref{fig:sum_diff} demonstrates how the minimal SoD cost of all stored schedules decreases with computation time. 
\begin{figure}[!t]
\centering
\input{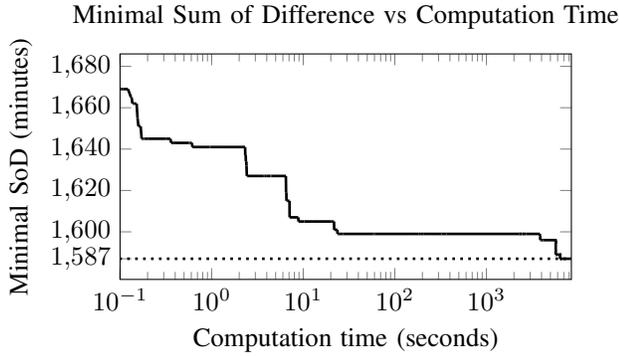}
\caption{The computation time versus minimal SoD cost among all stored schedules obtained within the corresponding computation time for Case Study 2. The dashed line represents the optimal SoD computed from the optimal schedule obtained by the mixed-integer program. }
\label{fig:sum_diff}
\end{figure}

The exact mixed-integer program cannot practically compute a solution if the number of demands exceeds about 50. On the other hand, the scheduling algorithm in this paper provides a sub-optimal schedule quickly for even a large demand set. To demonstrate this, we now consider 200 demands assigned randomly to the three routes with random deadlines in the interval $[40, 1540]$ minutes.
The algorithm finds the first feasible schedule at 0.2 seconds with Sum of Difference equal to 12662 minutes and, after 10.1 seconds, obtains a schedule with Sum of Difference equal to 12637 minutes. The lower bound for this set of demands is 7786 minutes. 
We note, however, that this is only a lower bound and it is likely that the optimal schedule has SoD cost significantly greater than this bound as was the case in the prior example with 27 UAVs.

\section{Conclusion}
In this paper, we studied the problem of dynamic scheduling in UAM networks with uncertain travel time. One main challenge is that nodes in a UAM network have limited parking spaces. As a result, a schedule for each UAV in the network has to be made before it takes off to ensure that a parking space is available upon arrival. 
Additionally, a mixed integer program is too time-expensive when the complexity of network and the size of demands increase. 

An exact schedule can sometimes be obtained from a mixed integer program, but this is not computationally tractable for larger networks and/or large sets of demands. Instead, we present a dynamic scheduling algorithm that is able to consider new demands over time and uses branch-and-bound heuristics to identify feasible but possibly sub-optimal schedules. In addition, we presented theoretical results establishing necessary conditions for a schedule to be feasible, and we further showed that these conditions become also sufficient conditions in certain cases.

Future work could consider scheduling for unforeseen disruptions such as one or more UAVs needing to reroute or land due to, e.g., adverse weather conditions.

\bibliography{Library} 
\bibliographystyle{IEEEtran}

\appendix

We  explain Algorithm \ref{alg:schedule_1} and \ref{alg:order_find} in this appendix. While Algorithm \ref{alg:order_find} is the main schedule-finding algorithm that provides a set of possible schedules, Algorithm \ref{alg:schedule_1} prepares the variables for Algorithm \ref{alg:order_find}, picks the best solution returned from Algorithm \ref{alg:order_find}, and transforms that into a schedule.

Algorithm \ref{alg:schedule_1} first computes $f_{j,v}$ by \eqref{eq:f_j_v} for all $j\in \mc J$ and $v\in V(R_j)$, based on the set of unassigned demands $\mc D_0 = \{D_j\}_{j\in \mc J}$. Recall that $f_{j,v}$ is the latest time for the UAV to arrive at node $v$ if $v \neq 0^{R_j}$ and is the latest departure time if $v = 0^{R_j}$ without passing its deadline $f_j$.
We initialize $DDL$, $PT$ and $RT$ as zero matrices of dimension $|\mc J|\times |\mc V|$.
We let $DDL(j,v) = f_{j,v}$, $PT(j,v) = m_v^{R_j}$ and $RT(j,v) = \ul{\mc M}^j_v$ for all $j$ and $v\in V(R_j)\backslash \{0^{R_j}\}$, and let $DDL(j,0^{R_j}) = f_{j,0^{R_j}}$, so that $DDL$, $PT$ ,and $RT$ represent the deadline for arrival/departure, possible time for occupying the nodes, and the shortest time for the UAV to travel from origin to node $v$ along its route, respectively.
It then calls the function $schedule$ in Algorithm \ref{alg:order_find} and picks the schedule that achieves the smallest SoD cost.

\begin{figure}[t]
 \removelatexerror
\begin{algorithm}[H]
\caption{Prepare Schedule}
\label{alg:schedule_1}
\begin{algorithmic}[1]
\Function{prepare\_schedule}{$\mc D_0 = \{D_j\}_{j\in \J}$}
\State Let $DDL$, $PT$, $RT$ be three $N_0$ by $\abs{\mc V}$ zero matrices.
\For{$j \in \mc J$}
\For{$v \in R_j $}
\State $DDL(j,v) := f_{j,v}$\hfill \Comment{$f_{j,v}$ from~\eqref{eq:f_j_v}} 
\State $PT(j,v) := m^{R_j}_{v}$ \Comment{$m^{R_j}_v$ from \eqref{eq:m_span}}
\State $RT(j,v) := \ul{\mc M}^j_v$ \Comment{$\ul{\mc M}^j_v$ from \eqref{eq:M_span}}
\EndFor
\EndFor
\State $RID :=  zeros(N_0,1)$, $RdpT := zeros(N_0,1)$,  
\Statex $SID := \emptyset$, $SdpT := \emptyset$ \hfill \Comment{Initialize}
\State $(SID, SdpT) :=$ \Call{schedule}{$\mc D_0, SID, \newline\hspace*{\fill}  SdpT, RID, RdpT, \abs{\mc D_0}, DDL, PT, RT$}
\State $\S_{new} := \emptyset $
\If{$SdpT \neq \emptyset$}
\State $i^* := \argmax_i \sum_j SdpT(j, i)$ \Comment{$i^*$ index of the optimal schedule}
\State Let $\mc S_{new}$ be the new schedule, with departure times $SdpT(:,i^*)$ ordered as $SID(:,i^*)$

\EndIf
\State \Return $\S_{new}$    
\EndFunction
\end{algorithmic}
\end{algorithm}
\vspace{-1cm}
\end{figure}

Algorithm \ref{alg:order_find} is a branch-and-bound algorithm that considers the set of candidate schedules as a rooted tree. 
Given the unassigned demands with ascending deadlines $\mc D_0 = \{D_j = (R_j,f_j)\}_{j \in \J_0}$, the algorithm visits the unassigned demands in descending order according to $f_j$ to pick the last journey in the schedule. $RdpT$  records the departure time while $RID$  records the index of the picked demand as in line \ref{alg_line:start}.

We then remove the demand from the unassigned list and delete the row of the corresponding demand from $DDL$, $PT$ and $RT$ before continuing with the current branch.
If the set of demands $\mc D_1 = \{D_j\}_{j \in \J_1}$, where $\mc J_1 \subseteq \mc J_0$, is already scheduled along the current branch, then for each node $v$, we can then compute $f^{node}_{v,c}$ as in Algorithm~\ref{alg:schedule} line~\ref{alg_line:c_start}--\ref{alg_line:c_end} for all $v\in \mc V$ and $c=1,\dots,C_v$ by substituting $\mc J_a = \mc J_1$ with line~\ref{alg_line:c_start}.
We implicitly assume that we will assign the journey to a parking spot where the earliest arrival time of the next journey at the same spot will be the latest among the $C_v$ parking spots for all assigned demands. We let $f^{node}_{v} = \max_{1\leq c \leq C_v} f^{node}_{v,c}$ for all $v \in \mc V\backslash S$ and let
$f'_{j,v} = \min (DDL(j,v), f^{node}_{v})$ for all $j \in \J_0 \backslash \J_1$ and all $v\in V(R_j)$. We then update $DDL$ so that
\begin{equation}
\label{eq:DDL_update1}
    DDL(j,0^{R_j})=\min_{v\in V(R_j)\backslash\{0\}}(f'_{j,v}-\ol{\mc M}^j_v)
\end{equation}
and 
\begin{equation}
\label{eq:DDL_update2}
    DDL(j,v) = DDL(j,0^{R_j})+\ol{\mc M}^j_v \, .
\end{equation}

The search-and-assign process is repeated until all demands are assigned or the branch is stopped by the pruning mechanism, which will be described later. If all demands are assigned, then 
the current branch, $RID$ and $RdpT$, will be stored into $SID$ and $SdpT$ as in line \ref{alg_line:branch_store}. The algorithm will then continue with another branch until all branches are considered (visited or pruned). The algorithm then returns $SID$ and $SdpT$. 

As mentioned in Section~\ref{sec:algorithm}, we have five main pruning conditions, while the first, second and the fourth conditions are easy to be realized by line~\ref{alg_line:cond_1}, \ref{alg_line:cond_2} and \ref{alg_line:cond_4}, respectively. We next explain the third and the fifth conditions in detail.

The third condition indicates that if the demand with latest deadline can be scheduled without interfering with any other demand, then we schedule it as the last journey to arrive at the destination. The following explains  line~\ref{alg_line:islast} in Algorithm \ref{alg:order_find}. 
Let ${j_0} = \argmax_{j \in \J_0\backslash\J_1} f_j$. For each $v\in V(R_{j_0})$, the number of unassigned demands that need to be considered is  
    \begin{equation}
    \label{eq:num_consider}
        num = \min\Big(C_v, \abs{\{D_j\}_{j \in\J_0\backslash\J_1 |v \in V(R_{j_0})}}-1\Big) \, .
    \end{equation}
By our parking-spot assigning assumption, if assigning $D_{j_0}$ as the last journey among all the unassigned demands will not interfere with any other demand, then 
    \begin{equation}
        \label{eq:is_last}
        \text{max}^k([f^{node}_{v,1},\dots,f^{node}_{v,C_v}]) \geq \text{max}^{k+1} (DDL(:,v))
    \end{equation}
for all $v\in V(R_{j_0})$ and $k = 1,\dots, num$, where we let $\text{max}^k(\cdot)$ represents the $k$'th largest number in the vector $(\cdot)$.

\begin{figure}[t!]
 \removelatexerror
\begin{algorithm}[H]
\caption{Schedule}
\label{alg:order_find}
\begin{algorithmic}[1]
\Function{schedule}{$\{D_{in,j}\}_{j\in \mc J}, SID,   SdpT,RID, \newline\hspace*{\fill} RdpT, N_0, DDL, PT, RT$}
\State  $\mc D_{in} = \{D_{in,j} = (R_{in,j}, f_{in,j})\}_{j\in \mc J}$
\State $N := \abs{\mc D_{in}}$
\State $\mc D = \{D_j = (R_j, f_j)\}_{j=1}^N :=$ the sorted sequence of $\mc D_{in}$, so that $\mc D = \mc D_{in}$ and $f_{j_1} \leq f_{j_2}$ if $j_1 \leq j_2$
\State $K := zeros(N,1)$
\State $K(j) := j'$ if and only if $D_j = D_{in,j'}$
\If{ $\exists j, v$ that $DDL(j,v)<0$} \label{alg_line:cond_1}
\State \Return $SID, SdpT$
\ElsIf{$\exists v$ that $\sum_{j} PT(j,v)<(\max_{j} DDL(j,v)- \min_{j} RT(j,v))$} \label{alg_line:cond_2}
\State \Return $SID, SdpT$
\ElsIf{$N==1$}\label{alg_line:branch_finish}
\State $dpT := DDL(N,0^{R_j})$ 
\State $RID(1) := K(1)$, $RdpT(1) := dpT$
\State $SID := [SID,RID]$, $SdpT := [SdpT,RdpT]$ \label{alg_line:branch_store}
\State \Return $SID, SdpT$
\Else
\If{$D_N$ can arrive at the each node at last without requiring any other UAV to depart earlier}  \hfill \Comment{\eqref{eq:is_last}} \label{alg_line:islast}
\State $i:=N$
\State \label{alg_line:start} $RID(N) := K(i)$,  $RdpT(N) :=  DDL(i,0^{R_i})$
\State $\mc D_2 := \{D_j\}_{j\neq i}$
\State $DDL_2 := DDL$,  $DDL_2(i,:):= [\,]$ 
\State $PT_2 := PT$, $PT_2(i,:) := [\,]$
\State $RT_2 := RT$,  $RT_2(i,:) := [\,]$
\State update $DDL_2$ according to \eqref{eq:DDL_update1} and \eqref{eq:DDL_update2}
\State compare $SID, SdpT$ with $\mc D_2$  \label{alg_line:compare}
\If{condition in~\eqref{eq:comp_stored} is satisfied} \label{alg_line:compare_result}
\State \label{alg_line:merge1} $RID(1:N-1) := SID_2$    
\State $RdpT(1:N-1) := SdpT_2$
\State $SID := [SID,RID]$
\State $SdpT := [SdpT,RdpT]$ \label{alg_line:merge2}
\Else
\State $[SID,SdpT] :=$ \Call{schedule}{$\mc D_{2}, SID, \newline\hspace*{\fill}  SdpT, RID,RdpT,  N_0, DDL_2, PT_2, RT_2$}
\EndIf \label{alg_line:end}
\Else
\For{$i \in \{N,\ldots, 1\}$}
\If{$\forall j$ such that $R_j = R_i$, $f_j \leq f_i$} \label{alg_line:cond_4}
\State same as line \ref{alg_line:start} -- \ref{alg_line:end}
\EndIf
\EndFor
\EndIf
\EndIf
\State \Return $SID,SdpT$
\EndFunction
\end{algorithmic}
\end{algorithm}
\vspace{-1cm}
\end{figure}

We now demonstrate how we can realize the last pruning condition with lines~\ref{alg_line:compare} and \ref{alg_line:compare_result}.
When searching along a branch, we will compare the current scheduled demands with stored assignments in $SID$ and $SdpT$. If the expected SoD cost along the current branch already exceeds any of the stored branches, the current branch will be abandoned. Consider the current unassigned demands $\{D_j\}_{j \in \J_0\backslash\J_1}$. If $SID(1:\abs{\J_0\backslash\J_1},col) = \J_0\backslash\J_1$ for some $col \in \mathbb N$, we then compute $f^{node}_{v,c}$ for all $v \in V(R_j)$ and $c=1,\dots,C_v$. 
Similarly, we can use $SID(:,col)$ and $SdpT(:,col)$ to obtain $\{\delta_j\}_{j \in \J_1}$ and compute $f^{node,S}_{v,c}$ for all $v \in V(R_j)$ and $c=1,\dots,C_v$. 
If
   \begin{equation}
   \label{eq:comp_stored}
       \text{max}^k([f^{node}_{v,1},\dots,f^{node}_{v,C_v}]) = \text{max}^k([f^{node,S}_{v,1},\dots,f^{node,S}_{v,C_v}])
   \end{equation}
for all $v \in V(R_j)$ and $k=1,\dots,num$, where $num$ can be computed as in \eqref{eq:num_consider}, is satisfied, then we can stop searching the current branch and instead use the stored schedule, as initiated on line~\ref{alg_line:compare_result}.

If more than one $col$ schedule satisfies \eqref{eq:comp_stored}, we pick the one with the least SoD cost, $col_{min}$. We let $SID_2 = SID(1:\abs{\J_0\backslash\J_1}, col_{min})$ and $SdpT_2 = SdpT(1:\abs{\J_0\backslash\J_1},col_{min})$. We then merge $SID_2$/$SdpT_2$ with $RID$/$RdpT$ as in line \ref{alg_line:merge1}--\ref{alg_line:merge2} and stop searching this branch.

\end{document}